\documentclass[11pt]{amsart}
\title{Brill-Noether varieties of $k$-gonal curves}
\author[N. Pflueger]{Nathan Pflueger}\address{Department of Mathematics, Brown University, Box
1917, Providence, RI 02912}\email{pflueger@math.brown.edu}
\date{\today}

\usepackage{url}
\usepackage{color}
\usepackage{xifthen}
\usepackage{tikz}
\usetikzlibrary{decorations.pathreplacing}
\usetikzlibrary{patterns}
\usepackage{mathabx}
\usepackage{youngtab}
\usepackage{amssymb}
\usepackage{array}

\usepackage{setspace}
\newtheorem{thm}{Theorem}[section]
\newtheorem{lemma}[thm]{Lemma}
\newtheorem{prop}[thm]{Proposition}

\newtheorem{cor}[thm]{Corollary}
\newtheorem{eg}[thm]{Example}

\newtheorem{conj}[thm]{Conjecture}
\newtheorem{defn}[thm]{Definition}

\newtheorem{rem}[thm]{Remark}
\newtheorem{sit}[thm]{Situation}
\newtheorem{qu}[thm]{Question}

\usepackage{pgfplots}

\newcommand{\ZZ}{\textbf{Z}}

\newcommand{\PP}{\textbf{P}}

\newcommand{\iou}[1][]{
    \ifthenelse{\equal{#1}{}}{{\color{blue}\{IOU\}}}
    {{\color{blue}\{IOU: #1\}}}
}

\newcommand{\ovr}{\overline{\rho}_{g,k}(d,r)}
\newcommand{\unr}{\underline{\rho}_{g,k}(d,r)}
\newcommand{\Gg}{\Gamma_{g,k,\ell}}

\DeclareMathOperator{\Pic}{Pic}

\DeclareMathOperator{\cd}{cd}
\DeclareMathOperator{\Spec}{Spec}
\DeclareMathOperator{\chr}{char }

\begin{document}
\maketitle

\begin{abstract}
We consider a general curve of fixed gonality $k$ and genus $g$. We propose an estimate $\ovr$ for the dimension of the variety $W^r_d(C)$ of special linear series on $C$, by solving an analogous problem in tropical geometry. Using work of Coppens and Martens, we prove that this estimate is exactly correct if $k \geq \frac15 g + 2$, and is an upper bound in all other cases. We also completely characterize the cases in which $W^r_d(C)$ has the same dimension as for a general curve of genus $g$.
\end{abstract}

\section{Introduction}

We are concerned in this paper with the varieties $W^r_d(C)$ of special linear series on a general curve of fixed genus $g$ and gonality $k$ over an algebraically closed field $K$ with mild restrictions on the characteristic. See Situation \ref{sit_main} for the necessary hypotheses.

For a general curve $C$ of genus $g$, we have $k = \lfloor \frac{g+3}{2} \rfloor$ and the Brill-Noether theorem \cite{gh} says that $\dim W^r_d(C)$ is equal to the Brill-Noether number
$$
\rho_g(d,r) = g - (r+1)(g-d+r),
$$

unless $\rho_g(d,r) < 0$, in which case $W^r_d(C) = \emptyset$.

Our objective is to compute $\dim W^r_d(C)$ for non-generic values of $k$. We propose a modification $\ovr$ of the Brill-Noether number, incorporating the value of $k$, and prove the following results. By convention, a negative estimate on $\dim W^r_d(C)$ means that $W^r_d(C)$ is empty.

\begin{thm} \label{thm_main} In Situation \ref{sit_main},
$\dim W^r_d(C) \leq \ovr$.
\end{thm}
In characteristic $0$, we may combine Theorem \ref{thm_main} with a lower bound from results of Coppens and Martens \cite{cm02}, to obtain the following sharp results.
\begin{thm} \label{thm_eq}
In Situation \ref{sit_main}, if $\chr K=0$ and either $k \leq 5$ or $k \geq \frac15 g + 2$, then $\dim W^r_d(C) = \overline{\rho}_{g,k}(d,r)$.
\end{thm}

\begin{thm} \label{thm_howgeneral}
In Situation \ref{sit_main}, if $\chr K = 0$ and  $\rho_g(d,r) \geq 0$, then $\dim W^r_d(C) = \rho_g(d,r)$ if and only if $r=0$, $g-d+r = 1$, or $g - k \leq d - 2r$.
\end{thm}

All three theorems use the following notation and hypotheses.

\begin{sit} \label{sit_main}
Fix nonnegative integers $g,k,r,d$ such that $2 \leq k \leq~\frac{g+3}{2}$ and $g-d+r > 0$. Let $K$ be an algebraically closed field, and $C$ be a general $k$-gonal curve of genus $g$ over $K$. Also assume that 
\begin{itemize}
\item If $k$ is odd, then $\chr K \neq 2$,
\item If $k=4$ or $10$, then $\chr K \neq 3$, and
\item If $k=6$, then $\chr K \neq 5$.
\end{itemize}
\end{sit}

The hypothesis $g-d+r > 0$ is harmless, since $g-d+r \leq 0$ would imply automatically that $W^r_d(C) = \Pic^d(C)$. The peculiar restrictions on the characteristic of $K$ arise in our proof when we must construct a \emph{tame} morphism of metrized complexes with certain properties. The characteristic $0$ assumption in Theorems \ref{thm_eq} and \ref{thm_howgeneral} is included because these make use of constructions from \cite{cm99} and \cite{cm02}, which assume this hypothesis.

\subsection{The estimate $\ovr$} The estimate $\ovr$ we refer to above is the following.

\begin{defn} \label{def_ovr}
Let $r'$ denote the minimum of $r$ and $g-d+r-1$. Define
$$
\ovr = \max_{\ell \in \{0,1,2,\cdots,r'\}} \left( \rho_g(d,r-\ell) - \ell k \right).
$$
\end{defn}

\begin{rem} \label{rem_whatsell}
The expression being maximized in this definition is a quadratic function in $\ell$, which takes its maximum (among all real numbers $\ell$) at $\ell_0 = \frac12(g-d+2r+1-k)$. It is therefore straightforward to write a closed-form expression for $\ovr$; see Remark \ref{rem_threecases} and the equation preceding it.
\end{rem}

Note that, by Riemann-Roch, $W^r_d(C) \cong W^{g-d+r-1}_{2g-2-d}(C)$. This explains the apparently strange appearance of $r'$ in Definition \ref{def_ovr}: it ensures that $\ovr$ is invariant under this duality.

We will relate $\ovr$ to $\dim W^r_d(C)$ by specializing to a metric graph $\Gg$ very similar to the metric graphs used in the tropical proof of the Brill-Noether Theorem \cite{cdpr}, but chosen with \textit{special} edge lengths. Making use of our description in \cite{pfl} of the Brill-Noether loci of such chains, we will see (Corollary \ref{cor_tropdim}) that $\dim W^r_d(\Gg)$ is equal to $\ovr$. The tropical lifting results of \cite{abbr}, combined with semicontinuity results for tropicalization, give Theorem \ref{thm_main}.

\subsection{The lower bound $\unr$} 
To deduce Theorems \ref{thm_eq} and \ref{thm_howgeneral} from Theorem \ref{thm_main}, we require a lower bound on $\dim W^r_d(C)$.

Coppens and Martens provide constructions in \cite{cm99} and \cite{cm02} that establish lower bounds on $\dim W^r_d(C)$ in Situation \ref{sit_main}, with an additional characteristic $0$ hypothesis. Their results imply the existence of irreducible components of $W^r_d(C)$ of dimension $\rho_g(d,r-\ell) - \ell k$ for certain values of $\ell$ in $\{0,1,\cdots,r'\}$. From their results, one can deduce the following lower bound on $\dim W^r_d(C)$, which coincides with the ``optimistic guess'' at the end of \cite{cm99} (they observe in \cite[Remark on p. 39]{cm02} that this optimistic guess is too small in general, though it is correct for $k \leq 5$).

\begin{defn}
Let $r' = \min(r,g-d+r-1)$. If $r' \geq 1$, let
$$
\unr = \max_{\ell \in \{0,1,r'-1,r'\}} \left( \rho_g(d,r-\ell) - \ell k \right).
$$
If $r' = 0$, let $\underline{\rho}_{g,k}(d,r) = \rho_g(d,r)$ (i.e. only allow $\ell = 0$).
\end{defn}

\begin{thm}[\cite{cm02}] \label{thm_cm}
In Situation \ref{sit_main}, if $\chr K = 0$ then 
$\dim W^r_d(C) \geq \unr$.
\end{thm}

Coppens and Martens do not state Theorem \ref{thm_cm} in the form that we have stated it here, but Theorem \ref{thm_cm} can be rapidly deduced from the main theorem in \cite{cm02}; the argument is provided in Section \ref{sec_proofs}.

\begin{rem} \label{rem_ab}
We will see in our analysis that it is useful to consider instead the quantites $g - \ovr$ $g - \unr$ (which are bounds on the \textit{codimension} of $W^r_d(C)$ in $\Pic^d(C)$). For fixed $k$, both of these quantities are functions of the following two variables.
\begin{eqnarray*}
a &=& r+1\\
b &=& g-d+r
\end{eqnarray*}
Observe that, for \emph{any} smooth curve $C$ of genus $g$, $W^r_d(C) \neq \Pic^d(C)$ if and only if both $a$ and $b$ are positive. It is useful to use these variables, rather than $r$ and $d$, for visualization purposes, such as in Figures \ref{fig_g20} and \ref{fig_disagree}; by the previous remark, the region of interest is the first quadrant. The maximizing value $\ell_0$ mentioned in Remark \ref{rem_whatsell} is easier to understand in these variables as well: it is $\ell_0 = \frac12( a+b-k)$.
\end{rem}

\begin{eg}
Consider the case $g=20$ in characteristic $0$. Since $\frac15 g + 2 = 6$, Theorem \ref{thm_eq} shows that $\dim W^r_d(C) = \overline{\rho}_{20,k}(d,r)$ in all cases. We can visualize how the existence of specific types of linear series varies with the gonality by plotting, for each possible gonality $k$, all points $(r+1,20-d+r)$ such that $W^r_d(C) \neq \emptyset$ for a general $k$-gonal curve of genus $20$. These plots are shown in Figure \ref{fig_g20}. We show only the points with both coordinates positive, since all other points correspond to values of $r,d$ for which $W^r_d(C) = \Pic^d(C)$.

\begin{figure}
\begin{tabular}{ccccc}
\includegraphics{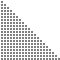}&
\includegraphics{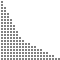}&
\includegraphics{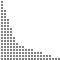}&
\includegraphics{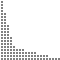}&
\includegraphics{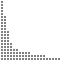}\\
$k=2$ & $k=3$ & $k=4$ & $k=5$ & $k=6$\\
\includegraphics{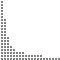}&
\includegraphics{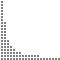}&
\includegraphics{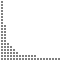}&
\includegraphics{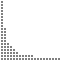}&
\includegraphics{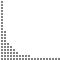}\\
$k=7$ & $k=8$ & $k=9$ & $k=10$ & $k=11$
\end{tabular}
\caption{
These plots show, for genus $20$ and every possible gonality $k$, the points $(g-d+r,r+1)$ such that $W^r_d(C)$ is nonempty for a general $k$-gonal curve $C$ in characteristic $0$.
} \label{fig_g20}
\end{figure}
\end{eg}

We consider the form of the results proved in \cite{cm02}, together with the analysis of the tropical version of the problem, as evidence for the following conjecture. Note in particular that we believe that the hypotheses on the characteristic of the field should be superfluous, although they are needed for technical reasons in our arguments.

\begin{conj}
If $g,k,r,d$ are nonnegative integers with $g-d+r > 0$ and $2 \leq k \leq \frac{g+3}{2}$, and $C$ is a general $k$-gonal curve of genus $g$ over an algebraically closed field, then $\dim W^r_d(C) = \ovr$.
\end{conj}

Beyond the dimension itself, one can ask about the dimensions of \textit{all} irreducible components of $W^r_d(C)$. The form of the upper bound, and the results of \cite{cm02}, suggest the following question.

\begin{qu}
Does every irreducible component of $W^r_d(C)$ have dimension equal to $\rho_g(d,r-\ell) - \ell k$ for some integer $\ell$?
\end{qu}

\subsection{The gap between $\unr$ and $\ovr$} \label{ss_gap}

The definitions of $\ovr$ and $\unr$, along with Remark \ref{rem_whatsell}, reveal that there are four cases where we can immediately conclude that $\unr = \ovr$ and compute $\dim W^r_d(C)$ (for curves over characteristic $0$ fields). These are the cases where the maximum in Definition \ref{def_ovr} occurs at $\ell = 0,1,r'-1,r'$, respectively (note that these cases may overlap in situations where the maximum occurs for two integers $\ell$, or when $r' \leq 2$). Here $r' = \min(r,g-d+r-1)$, as in Definition \ref{def_ovr}.

The issue of comparing $\unr$ to $\ovr$ therefore is confined to the situation where the maximum in Definition \ref{def_ovr} occurs for a value $\ell$ such that $2 \leq \ell \leq r'-2$. In the variables $(a,b) = (r+1,g-d+r)$, this is equivalent to the inequalities $2 \leq \frac{a+b-k}{2} \leq \min(a-3,b-3)$; equivalently $a+b \geq 4+k$ and $|a-b| \leq k-6$ (where we assume $a,b > 0$, as is implied by Situation \ref{sit_main}). See Figure \ref{fig_disagree}.

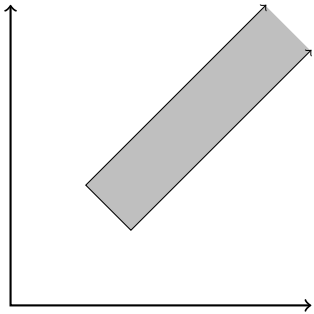
\begin{figure}
% This is the case k = 9.
\begin{tikzpicture}[scale=0.2]
\draw[<->,thick] (0,20) -- (0,0) -- (20,0);
\draw[<->,fill=lightgray] (17,20) -- (5,8) -- (8,5) -- (20,17);
\end{tikzpicture}
\caption{The region in which $\unr < \ovr$, where the coordinates are $a = r+1$ and $b = g-d+r$. It is bounded by the lines $a+b = 4+k$ and $a-b = \pm (k-6)$, and empty for $k \leq 5$.} \label{fig_disagree}
\end{figure}

Our proof of Theorem \ref{thm_eq} proceeds by showing that as long as $k \geq \frac15 g + 2$ or $k \leq 5$, \textit{all} $d,r$ corresponding to points in this region give $\ovr < 0$. The following example shows that even for smaller $k$, the ambiguous cases are nonetheless quite sparse, even in the worst case.

\begin{eg}
Consider the case $g=1000$. For each possible value of the gonality $k$, we can enumerate all possible pairs of positive integers $r,d$ with $d \leq g-1$ and $\ovr \geq 0$ (we will restrict to $d \leq g-1$, since the other cases are symmetric via Riemann-Roch), and count the proportion of these for which $\unr < \ovr$. The largest such proportion is $0.042$, for gonality $k=40$; in that case there are $13123$ possible pairs where $\ovr \geq 0$, of which $552$ satisfy $\ovr < \unr$. Of those $552$ cases, there are only $69$ cases where $\unr < 0$. In all other cases there is no ambiguity about whether $W^r_d(C)$ is nonempty, but only about its exact dimension.
\end{eg}

\subsection{Outline of the paper}

In Section \ref{sec_lifting}, we define the metric graphs $\Gg$ where we formulate the tropical analog we will study, and prove that these graphs can be lifted to algebraic curves using results of \cite{abbr}. We also cite the necessary statements of our \cite{pfl} that translate the computation of $\dim W^r_d(\Gg)$ into the analysis of certain \textit{displacement tableaux}. This analysis is performed in Section \ref{sec_disp}, which is purely combinatorial. In section \ref{sec_proofs} we deduce the theorems stated in the introduction.

\begin{rem}
An alternative to the tropical approach that we use in Section \ref{sec_lifting} would be to instead consider limit linear series on a chain of elliptic curves, where the two attachment points on each elliptic curve differ by $k$-torsion. Such an approach would reduce to the same analysis of displacement tableaux.
\end{rem}

\section{Lifting metric graphs} \label{sec_lifting}

In this section, we construct the metric graphs which will be the basis of our argument, and prove the necessary lifting results to algebraic curves.

\begin{defn}
Let $\Gamma_{g,k,\ell}$ denote the metric graph formed by connecting $g$ cycles in a chain as in Figure \ref{fig_chain}, with edge lengths as follows.
\begin{itemize}
\item The length of the top (clockwise) edge from $w_i$ to $w_{i+1}$ is $\ell$.
\item The length of the bottom (counterclockwise) edge from $w_i$ to $w_{i+1}$ is $k-\ell$.
\end{itemize}
\end{defn}

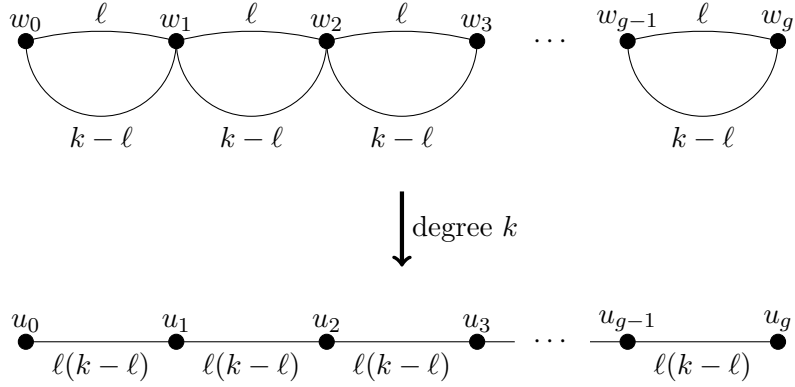
\begin{figure}
\begin{tikzpicture}[scale=2]
\draw (0,0) arc (105:75:1.931);
\draw (1,0) arc (0:-180:0.5);
\draw (1,0) arc (105:75:1.931);
\draw (2,0) arc (0:-180:0.5);
\draw (2,0) arc (105:75:1.931);
\draw (3,0) arc (0:-180:0.5);
\draw (3.5,0) node {$\cdots$};
\draw (4,0) arc (105:75:1.931);
\draw (5,0) arc (0:-180:0.5);

\foreach \x in {0,1,2,4} {
\draw (\x+0.5,0.05) node[above] {$\ell$};
\draw (\x+0.5,-0.5) node[below] {$k - \ell$};
\draw (\x+0.5,-2) node[below] {$\ell(k-\ell)$};
}

\draw[->,ultra thick] (2.5,-1) -- (2.5,-1.5);
\draw (2.5,-1.25) node[right] {degree $k$};
\draw (0,-2) -- (3.25,-2);
\draw[fill] (0,-2) circle[radius=0.05] node[above] {$u_0$};
\draw[fill] (1,-2) circle[radius=0.05] node[above] {$u_1$};
\draw[fill] (2,-2) circle[radius=0.05] node[above] {$u_2$};
\draw[fill] (3,-2) circle[radius=0.05] node[above] {$u_3$};
\draw (3.5,-2) node {$\cdots$};
\draw (3.75,-2) -- (5,-2);
\draw[fill] (4,-2) circle[radius=0.05] node[above] {$u_{g-1}$};
\draw[fill] (5,-2) circle[radius=0.05] node[above] {$u_g$};

\draw[fill] (0,0) circle[radius=0.05] node[above] {$w_0$};
\draw[fill] (1,0) circle[radius=0.05] node[above] {$w_1$};
\draw[fill] (2,0) circle[radius=0.05] node[above] {$w_2$};
\draw[fill] (3,0) circle[radius=0.05] node[above] {$w_3$};
\draw[fill] (4,0) circle[radius=0.05] node[above] {$w_{g-1}$};
\draw[fill] (5,0) circle[radius=0.05] node[above] {$w_g$};
\end{tikzpicture}
\caption{The chain of cycles $\Gamma_{g,k,\ell}$ and the harmonic morphism to an interval.} \label{fig_chain}
\end{figure}

\begin{defn}
Let $k,\ell$ be positive integers with $\ell < k$, and let $p$ be either a prime number or $0$. Call the triple $(p,k,\ell)$ \emph{admissible} if $\gcd(\ell,k) = \gcd(\ell,p) = \gcd(k-\ell,p) = 1$.
\end{defn}

This notion of admissibility underlies the restrictions on the characteristic of the field in Situation \ref{sit_main}.

\begin{lemma} \label{lem_lexists}
Let $p,k$ be integers, with $k \geq 2$ and either $p=0$ or $p$ prime. There exists an integer $\ell \in \{1,2,\cdots,k-1\}$ such that $(p,k,\ell)$ is admissible if and only if 
$$
(p,k) \not\in \{(2,k):\ k \mbox{ odd}\} \cup \{(3,4),(3,10),(5,6)\}.
$$
\end{lemma}
\begin{proof}
Assume that $(p,k)$ are chosen, not belonging to any of these four cases stated in the lemma. Let $\ell$ be an integer, chosen as follows.
\begin{enumerate}
\item If $k \not\equiv 1 \pmod{p}$, let $\ell = 1$ (this case includes the case $p=0$).
\item If $k \equiv 1 \pmod{p}$ and $k$ is odd, let $\ell = 2$.
\item If $k \equiv 1 \pmod{p}$ and $k$ is even, choose $\ell$ according to the following grid. Observe that our assumptions imply, in this case, that $p$ is an odd prime.
\begin{center}
{
\setlength\extrarowheight{5pt}
$\begin{array}{|l |l l l |}\hline
& p=3 & p=5 & p > 5 \\\hline
k \equiv 0 \pmod{4} & \ell = \frac12 k - 3 & \ell = \frac12 k - 1 & \ell = \frac12 k - 1\\
k \equiv 2 \pmod{4} & \ell = \frac12 k - 6 & \ell = \frac12 k - 4 & \ell = \frac12 k - 2\\\hline
\end{array}$
}
\end{center}
\end{enumerate}
In cases (1) and (2), it follows immediately that $(p,k,\ell)$ is admissible (note in case (2) that our assumption has ruled out $p=2$). In case $3$, we must verify that $\ell \not\equiv 0,1 \pmod{p}$, $\gcd(\ell,k) = 1$, and $\ell \geq 1$. This is a routine verification in each of the six sub-cases.

Conversely, it is straightforward to check that no such $\ell$ exists in the four cases mentioned in the statement.
\end{proof}

\begin{lemma} \label{lem_lift}
Let $K$ be a complete algebraically closed non-Archimedean field of characteristic $p$ (possibly $p=0$), with nontrivial valuation and residue field of the same characteristic. Let $\ell$ be a positive integer such that $(p,k,\ell)$ is admissible. There exists a smooth projective curve $C$ over $K$ with the following properties:
\begin{enumerate}
\item The genus of $C$ is $g$;
\item There is a degree $k$ map $C \rightarrow \PP^1$;
\item The minimal skeleton of $C$ is isometric to $\Gg$.
\end{enumerate}
\end{lemma}

\begin{proof}
We apply the results of \cite{abbr} on lifting harmonic morphisms of metrized complexes to morphisms of algebraic curves.

Let $I$ denote a metric graph consisting of vertices $u_0,u_1,\cdots,u_g$ joined (in that order) into a path, with all edge lengths equal to $\ell(k-\ell)$ (see Figure \ref{fig_chain}).

Define a piecewise-linear map $\pi: \Gg \rightarrow I$ by $\pi(w_i) = u_i$, with $\pi$ linear between these vertices. This map of metric graphs has integer slopes, with expansion factor $k-\ell$ on the top (clockwise) edges $w_i w_{i+1}$, and expansion factor $\ell$ on the bottom (counterclockwise) edges $w_i w_{i+1}$. Since $(p,k,\ell)$ is admissible, all expansion factors are positive integers, none divisible by the characteristic of $K$, such that at any vertex of $\Gg$, the expansion factors along any outward rays mapping to the leftward ray in $I$ add up to $k$ (and likewise for outward rays mapping to the rightward ray in $I$). Therefore $\pi$ is a \textit{finite harmonic} map of metric graphs, of degree $k$ \cite[Definitions 2.4,\ 2.6]{abbr}.

This harmonic map of metric graphs gives rise to harmonic morphism of \textit{metrized complexes} \cite[Definition 2.19]{abbr} as follows: to each vertex $v$ in either $\Gg$ or $I$ we associate a curve $C_v$ isomorphic to $\PP^1$ over the residue field of $K$; for each vertex $w_i$ of $\Gg$ we associate a degree $k$ morphism $C_{w_i} \rightarrow C_{u_i}$ defined by a rational function with divisor $(k-\ell) \cdot A + \ell \cdot B - (k-\ell) \cdot C - \ell \cdot D$, where $A,B,C,D$ denote any four rational points of $C_{w_i}$. We identify the two points $A,B$ of $C_{w_i}$ with the two tangent directions from $w_i$ that point to the left (if they exist), and we identify with $C,D$ the tangent directions that points to the right (if they exist). We identify the rational point $(0)$ of $C_{u_i}$ with the leftward tangent direction in $I$, and the rational point $(\infty)$ with the rightward tangent direction.

The morphism $\pi'$ is \textit{tame} \cite[Definition 2.21]{abbr}, since none of the expansion factors of $\pi$ are divisible by the characteristic of the residue field of $K$. Therefore \cite[Proposition 7.15]{abbr} implies the existence of a degree $k$ morphism $C \rightarrow \PP^1$ of smooth projective curves over $K$ specializing to the morphism $\pi'$ of metrized complexes. In particular, the minimal skeleton of $C$ is isometric to $\Gg$, and $C$ has a degree $k$ map to $\PP^1$, as desired.
\end{proof}

Lemma \ref{lem_lift} makes it possible to deduce results about algebraic curves from the tropical Brill-Noether theory of the graph $\Gg$, i.e. of the loci $W^r_d(\Gg)$, which can be defined in terms of the Baker-Norine rank of divisors on metric graphs (see \cite{gk} or \cite{mz}). Our \cite{pfl} develops the tools necessary for this analysis, applicable to any chain of cycles with any edge lengths. For convenience, we summarize the needed facts in the context necessary for this paper.

The structure of the tropical Brill-Noether loci $W^r_d(\Gg)$ depends on the \textit{torsion profile} of the metric graph $\Gg$, which is defined as the tuple $m_2, m_3, \cdots, m_{g-1}$, where $m_i$ is the generator of the arithmetic progression $\{m \in \ZZ: m \cdot w_{i-1} \sim m \cdot w_i \}$ \cite[Definition 1.9]{pfl} (in \cite{pfl}, the torsion profile also includes a final number $m_g$, but this number is irrelevant in the context of this paper). For the metric graph $\Gg$, where $\ell$ is chosen relatively prime to $k$ (which is part of the definition of $(p,k,\ell)$ being admissible), all of the numbers $m_i$ are equal to $k$. We therefore make the following definition, which is a simplified version of \cite[Definition 2.1]{pfl} for the case where all of the $m_i$ equal the same number $k$.

\begin{defn} \label{def_disp}
Let $a,b$ be positive integers, and let $\lambda$ be the set $\{1,2,\cdots,b\} \times \{1,2,\cdots,a\}$. A \emph{$k$-uniform displacement tableau} on $\lambda$ is a function $t: \lambda \rightarrow \ZZ_{>0}$ subject to the following two conditions.
\begin{enumerate}
\item $t(x+1,y) > t(x,y)$ and $t(x,y+1) > t(x,y)$ whenever both sides of the inequality are defined.
\item If $t(x,y) = t(x',y')$, then $x-y \equiv x'-y' \pmod{k}$.
\end{enumerate}
The elements of $\lambda$ will be called ``boxes'' and the integer $t(x,y)$ will be called the ``label'' of the box $(x,y)$.
\end{defn}

\begin{eg}
The following is a $3$-uniform displacement tableau, with $a = b = 3$. The labels $3$ and $5$ are repeated, which is permitted since the occurrences are in boxes of lattice distance $k=3$ from each other.
$$
\young(567,345,123)
$$
\end{eg}

\begin{lemma} \label{l_td}
Fix integers $r,d$, and let 
$$
\lambda = \{1,2,\cdots,g-d+r\} \times \{1,2,\cdots,r+1\}.$$
Suppose that $\ell$ is relatively prime to $k$. The locus $W^r_d(\Gg)$ is nonempty if and only if there exists a $k$-uniform displacement tableau on $\lambda$ with at most $g$ distinct labels. If it is nonempty, then its dimension is equal to $g - | t(\lambda) |$, where $t$ has the fewest possible number of distinct labels in such a tableau.
\end{lemma}

\begin{proof}
If such a tableau exists, then we may assume without loss of generality that the labels used are all chosen from $\{1,2,\cdots, g\}$, by applying an order-preserving function to the labels used. The result is now \cite[Corollary 3.7]{pfl}, applied to the case where $\lambda$ is rectangular and all the torsion orders are equal to $k$.
\end{proof}

\begin{cor} \label{cor_punchline}
Let $C$ be a curve of the form guaranteed to exist by Lemma \ref{lem_lift}. Then the dimension of $W^r_d(C)$ is bounded above by the maximum number of omitted labels from $\{1,2,\cdots,g\}$ in a $k$-uniform displacement tableau on $\{1,2,\cdots,g-d+r\} \times \{ 1,2,\cdots, r+1 \}$.
\end{cor}

\begin{proof}
This follows from \cite[Proposition 5.1]{pfl} applied to the case of a rectangular partition, which states that the dimensions of the Brill-Noether loci of $C$ are bounded above by the dimensions of the Brill-Noether loci of $\Gg$.
\end{proof}

We have therefore reduced our task to the analysis of $k$-uniform displacement tableaux, which we undertake in the following section.

\section{Uniform displacement tableaux} \label{sec_disp}

Our purpose in this section is to compute the minimum number of symbols needed to form a $k$-uniform displacement tableau on a rectangular partition. The answer is given by the following function.

\begin{defn} \label{def_cd}
Let $a,b,k$ be positive integers with $k \geq 2$. Let $\cd(a,b,k)$ denote the fewest number of distinct symbols used in a $k$-uniform displacement tableau on an $a \times b$ rectangular partition. Define $\delta(a,b,k)$ by
$$
\delta(a,b,k) = \min_{\ell \in \{0,1,\cdots,\min(a,b)-1\}} \left( (a-\ell)(b-\ell) + k \ell \right).
$$
\end{defn}

The name $\cd$ is chosen since this quantity will be used to bound the codimensions of Brill-Noether loci. Indeed, note that by definition, $$\ovr = g - \delta(r+1,g-d+r,k).$$

Observe that $\cd(a,b,k) = \cd(b,a,k)$ and $\delta(a,b,k) = \delta(b,a,k)$. Therefore it suffices to consider the case $a \leq b$, which will simplify some statements.

\begin{rem} \label{rem_threecases}
The function being minimized in the definition of $\delta(a,b,k)$ is a quadratic function in $\ell$, which achieves its minimum (for real values of $\ell$) at $\ell = \frac12 (a+b-k)$. If $\ell$ is constrained to integer values, the minimum is attained by both the floor or the ceiling of $\frac12 (a+b-k)$. An elementary calculation shows that if $a \leq b$, then
$$
\delta(a,b,k) = \begin{cases}
(k-1)(a-1) + b & \mbox{if $k \leq b-a+3$},\\
ab - \left\lfloor \left( \frac{a+b-k}{2} \right)^2 \right\rfloor & \mbox{if $b-a+1 \leq k \leq a+b+1$},\\
ab & \mbox{if $k \geq a+b-1.$}
\end{cases}
$$
The overlap of the hypotheses for these cases reflects the fact that there are two minimizing values of $\ell$ when $a+b-k$ is odd. If $a \geq b$, we obtain similar formulas by symmetry.
\end{rem}

\begin{prop} \label{p_cd}
For any positive integers $a,b,k$ with $k \geq 2$, $$\cd(a,b,k) = \delta(a,b,k).$$
\end{prop}

We prove Proposition \ref{p_cd} by proving inequalities in both directions. In fact, only the first inequality is needed for this paper's main results, but we present the second inequality to emphasize that our result is the strongest possible using this method.

\begin{lemma} \label{l_cdlower}
For any positive integers $a,b,k$ with $k \geq 2$, $$\cd(a,b,k) \geq \delta(a,b,k).$$
\end{lemma}

\begin{proof} 
We assume without loss of generality that $a \leq b$.

Let $\lambda$ denote the rectangular partition $\{1,2,\cdots,b\} \times \{1,2,\cdots,a\}$, and let $t: \lambda \rightarrow \ZZ_{>0}$ be any $k$-uniform displacement tableau on $\lambda$. 
We consider three cases separately (corresponding to the three cases of Remark \ref{rem_threecases}). In each case, we will give a subset $S$ of the boxes of $\lambda$, with $\delta(a,b,k)$ elements, such that each box must have a different value in $t$. This will imply the statement of the lemma.

Throughout, we will say that a box $(x',y')$ \emph{dominates} another box $(x,y)$ if $x' \geq x$ and $y' \geq y$. If so, then $(x',y')$ and $(x,y)$ cannot have the same label.

\textit{Case 1:} $k \geq a+b-1$. For any box $(x,y) \in \lambda$, $1-a \leq x-y \leq b-1$. Therefore, if $(x,y),(x',y')$ are two boxes with $x-y \equiv x'-y' \pmod{k}$, then $x-y = x'-y'$ and one box dominates the other. Therefore we take $S$ to be all of $\lambda$; all boxes must have distinct labels, so $\cd(a,b,k) = \delta(a,b,k) = ab$.

\textit{Case 2:} $k \leq b-a+2$. Define $S \subseteq \lambda$ to be the disjoint union $S_1 \cup S_2$, where
\begin{eqnarray*}
S_1 &=& \left\{ (x,y):\ 1 \leq y \leq a-1 \textrm{ and } y \leq x \leq y+k-1 \right\} \\
S_2 &=&  \left\{ (x,a):\ a \leq x \leq b \right\}.
\end{eqnarray*}
An example is shown in Figure \ref{fig_s}. Any box $(x,y)$ in $S_1$ satisfies $0 \leq x-y \leq k-1$, hence no two boxes in $S_1$ can have the same label, by the same reasoning as in case $1$. Any two boxes in $S_2$ must have different labels since one dominates the other. Finally, if $(x,y) \in S_1$ and $(x',a) \in S_2$ satisfy $x-y \equiv x'-a \pmod{k}$, then $(x',a)$ either dominates or is equal to $(x-y+a, a)$ (since $x' \geq a$, $x' \equiv x-y + a \pmod{k}$, and $0 \leq x-y \leq k-1$), which dominates $(x,y)$. Therefore no box of $S_1$ can have the same label as a box of $S_2$. It follows that the boxes of $S$ all have distinct labels in $t$. The size of $S$ is $k(a-1) + b-a+1 = (k-1)(a-1) + b = \delta(a,b,k)$, so $\cd(a,b,k) \geq \delta(a,b,k)$.

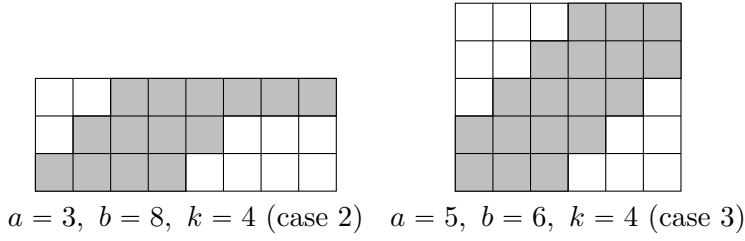
\begin{figure}
\begin{tabular}{cc}
\begin{tikzpicture}[scale=0.5]
\draw (0,0) rectangle (8,3);
\draw[fill=lightgray] (0,0) -- (4,0) -- (4,1) -- (5,1) -- (5,2) -- (8,2) -- (8,3) -- (2,3) -- (2,2) -- (1,2) -- (1,1) -- (0,1) -- cycle;
\foreach \x in {1,2,...,7} {\draw (\x,0) -- (\x,3);}
\foreach \y in {1,2} {\draw (0,\y) -- (8,\y);}
\end{tikzpicture}
&
\begin{tikzpicture}[scale=0.5]
\draw (0,0) rectangle (6,5);
\draw[fill=lightgray] (0,0) -- (3,0) -- (3,1) -- (4,1) -- (4,2) -- (5,2) -- (5,3) -- (6,3) -- (6,5) -- (3,5) -- (3,4) -- (2,4) -- (2,3) -- (1,3) -- (1,2) -- (0,2) -- cycle;
\foreach \x in {1,2,...,5} {\draw (\x,0) -- (\x,5);}
\foreach \y in {1,2,3,4} {\draw (0,\y) -- (6,\y);}
\end{tikzpicture}
\\
$a=3,\ b=8,\ k=4$ (case 2)
& 
$a=5,\ b=6,\ k=4$ (case 3)
\end{tabular}
\caption{Examples of the regions $S$ is cases $2$ and $3$ of the proof of Lemma \ref{l_cdlower}.} \label{fig_s}
\end{figure}

\textit{Case 3:} $b-a+3 \leq k \leq a+b-2$. Define
$$
S = \left\{ (x,y) \in \lambda:\ - \left\lceil \frac{k-1-(b-a)}{2} \right\rceil \leq x-y \leq \left\lfloor \frac{k-1+(b-a)}{2} \right\rfloor \right\}.
$$

Since $(b-a) \leq k-3$, the lower bound on $x-y$ in this definition is indeed negative. An example is shown in Figure \ref{fig_s}.

Observe that the difference between the upper bound and the lower bound for $x-y$ in the definition of $S$ is exactly $k-1$. Therefore any two boxes in $S$ with values of $x-y$ that are congruent modulo $k$ must in fact have equal values of $x-y$, hence one dominates the other. It follows that $\cd(a,b,k) \geq |S|$. It remains to calculate the size of the set $S$.

The complement of $S$ can be written as the union of two disjoint parts.
\begin{eqnarray*}
T_1 &=& \left\{ (x,y) \in \lambda:\ y \geq x + 1 + \left\lceil \frac{k-1-(b-a)}{2} \right\rceil \right\}\\
T_2 &=& \left\{ (x,y) \in \lambda:\ x \geq y + 1 + \left\lfloor \frac{k-1+(b-a)}{2} \right\rfloor \right\}
\end{eqnarray*}

Both $T_1$ and $T_2$ are triangular regions in the upper left and lower right corners of $\lambda$. By determining the side lengths of these two regions, it follows that

\begin{eqnarray*}
|T_1| &=& \binom{ \lfloor (a+b-k+1)/2 \rfloor }{2},\\
|T_2| &=& \binom{ \lceil (a+b -k+1)/2 \rceil }{2}.
\end{eqnarray*}

Let $z = \frac{a+b-k+1}{2}$ and $\varepsilon = z - \lfloor z \rfloor$. Then

\begin{eqnarray*}
|T_1| + |T_2| &=& \frac12 (z - \varepsilon)(z - 1 - \varepsilon) + \frac12 (z + \varepsilon)(z-1+\varepsilon)\\
&=& z(z-1) + \varepsilon^2\\
&=& \left(z-\frac12\right)^2 + \varepsilon^2 - \frac14\\
&=& \left \lfloor \left( \frac{a+b-k}{2} \right)^2 \right\rfloor
\end{eqnarray*}

Therefore $|S| = ab - \left \lfloor \left( \frac{a+b-k}{2} \right)^2 \right\rfloor = \delta(a,b,k)$ (by the formula in Remark \ref{rem_threecases}). Since all the boxes of $S$ must have distinct labels, it follows that $\cd(a,b,k) \geq \delta(a,b,k)$.
\end{proof}

\begin{lemma} \label{l_cdupper}
For any positive integers $a,b,k$ with $k \geq 2$,
$$
\cd(a,b,k) \leq \delta(a,b,k).
$$
\end{lemma}
\begin{proof}
We assume without loss of generality that $a \leq b$. The case $a=1$ can be checked directly (both $\cd(1,b,k)$ and $\delta(1,b,k)$ are equal to $b$), so we also assume that $a \geq 2$.

We give an explicit construction for $k$-uniform displacement tableaux with exactly $\delta(a,b,k)$ symbols.

\textit{Case 1:} $k \geq a+b-1$. In this case, $\delta(a,b,k) = ab$, so we can take $t$ to be any tableau with all distinct symbols.

\textit{Case 2:} $k \leq a+b-2$. Define $\ell = \min \left( a-1,\ \left\lceil \frac{a+b-k}{2} \right\rceil \right)$; this value is chosen so that $\delta(a,b,k) = (a-\ell)(b-\ell) + k \ell$. Note that our assumptions imply that $\ell \geq 1$. We will construct a tableau $t$ on $\lambda$ with exactly $(a-\ell)(b-\ell) + k \ell$ distinct symbols, which will establish that $\cd(a,b,k) \leq \delta(a,b,k)$.

\textit{Definition of $t$}. Informally, $t$ will be defined by first dividing the partition $\lambda$ into strips of width $1$ and height $a-\ell$ (with the top row of strips truncated to a smaller height), and filling each strip with consecutive numbers from bottom to top. In each row of strips, we fill the strips in order, from left to right. In order to minimize the distinct symbols used, we begin filling each row of strips before the previous row is complete: the first strip of one row is filled simultaneously with the $(k+\ell-a+1)$th strip of the row below.

Formally, define first a function $t: \lambda \rightarrow \ZZ_{>0}$ as follows.

\begin{eqnarray*}
t(x,y) &=& (a-\ell) \cdot \left[ (x-1) + q (k+\ell-a) \right] + r,\\
\mbox{where } q &=& \left\lfloor \frac{y-1}{a-\ell} \right\rfloor\\
\mbox{and } r &=& y - q \cdot (a-\ell).
\end{eqnarray*}

An example is shown in Figure \ref{fig_construction}.

\begin{figure}
\begin{tikzpicture}[scale=0.7]
\draw (0,0) -- (7,0) -- (7,7) -- (0,7) -- (0,0);
\draw (1,0) -- (1,7);
\draw (2,0) -- (2,7);
\draw (3,0) -- (3,7);
\draw (4,0) -- (4,7);
\draw (5,0) -- (5,7);
\draw (6,0) -- (6,7);
\draw (0,3) -- (7,3);
\draw (0,6) -- (7,6);
\draw[fill=lightgray] (0,6) rectangle (1,7);
\draw[fill=lightgray] (3,3) rectangle (4,6);
\draw[fill=lightgray] (6,0) rectangle (7,3);
\draw (0.500000,0.500000) node {$1$};
\draw (0.500000,1.500000) node {$2$};
\draw (0.500000,2.500000) node {$3$};
\draw (0.500000,3.500000) node {$10$};
\draw (0.500000,4.500000) node {$11$};
\draw (0.500000,5.500000) node {$12$};
\draw (0.500000,6.500000) node {$19$};
\draw (1.500000,0.500000) node {$4$};
\draw (1.500000,1.500000) node {$5$};
\draw (1.500000,2.500000) node {$6$};
\draw (1.500000,3.500000) node {$13$};
\draw (1.500000,4.500000) node {$14$};
\draw (1.500000,5.500000) node {$15$};
\draw (1.500000,6.500000) node {$22$};
\draw (2.500000,0.500000) node {$7$};
\draw (2.500000,1.500000) node {$8$};
\draw (2.500000,2.500000) node {$9$};
\draw (2.500000,3.500000) node {$16$};
\draw (2.500000,4.500000) node {$17$};
\draw (2.500000,5.500000) node {$18$};
\draw (2.500000,6.500000) node {$25$};
\draw (3.500000,0.500000) node {$10$};
\draw (3.500000,1.500000) node {$11$};
\draw (3.500000,2.500000) node {$12$};
\draw (3.500000,3.500000) node {$19$};
\draw (3.500000,4.500000) node {$20$};
\draw (3.500000,5.500000) node {$21$};
\draw (3.500000,6.500000) node {$28$};
\draw (4.500000,0.500000) node {$13$};
\draw (4.500000,1.500000) node {$14$};
\draw (4.500000,2.500000) node {$15$};
\draw (4.500000,3.500000) node {$22$};
\draw (4.500000,4.500000) node {$23$};
\draw (4.500000,5.500000) node {$24$};
\draw (4.500000,6.500000) node {$31$};
\draw (5.500000,0.500000) node {$16$};
\draw (5.500000,1.500000) node {$17$};
\draw (5.500000,2.500000) node {$18$};
\draw (5.500000,3.500000) node {$25$};
\draw (5.500000,4.500000) node {$26$};
\draw (5.500000,5.500000) node {$27$};
\draw (5.500000,6.500000) node {$34$};
\draw (6.500000,0.500000) node {$19$};
\draw (6.500000,1.500000) node {$20$};
\draw (6.500000,2.500000) node {$21$};
\draw (6.500000,3.500000) node {$28$};
\draw (6.500000,4.500000) node {$29$};
\draw (6.500000,5.500000) node {$30$};
\draw (6.500000,6.500000) node {$37$};
\end{tikzpicture}
\caption{The construction in the proof of lemma \ref{l_cdupper}, in the case $a=7,\ b=7,\ k=6$. The tableau is constructed by filling $3 \times 1$ strips in sequence. Although the labels of the tableau go up to $37$, there are four values that do not occur ($32,33,35,36$), so there are exactly $33$ distinct labels.} \label{fig_construction}
\end{figure}

\textit{Validity of $t$}. We first verify that this function does in fact give a $k$-uniform displacement tableau. Clearly $t(x+1,y) > t(x,y)$. The value $t(x,y+1)-t(x,y)$ is either $1$ (if $y \not \equiv 0 \pmod{a-\ell}$), or $(a-\ell)(k+\ell-a) - (a-\ell-1)$ otherwise. The latter quantity is equal to $(a - \ell)( k + \ell - a -1) + 1$, which is positive since $k + \ell -a - 1 \geq 0$ and $\ell \leq a-1$ (these follow from our choice of the number $\ell$). Hence $t(x,y)$ is strictly increasing in rows and columns; it remains to show that any two boxes $(x,y)$ with the same label in $t$ have values of $x-y$ that are congruent modulo $k$.

Since $1 \leq r \leq a - \ell$, two boxes satisfy $t(x,y) = t(x',y')$ if and only if $r = r'$ and $(x-1) + q(k+\ell-a) = (x'-1) + q'(k+\ell-a)$ (where $q',r'$ are obtained from $y'$ in the same manner as $q,r$ are obtained from $y$). The latter equation implies that 
$$
x - q (a - \ell) \equiv x' - q'(a-\ell) \pmod{k},
$$
and subtracting $r = r'$ from both sides gives $x-y \equiv x' - y' \pmod{k}$. So $t$ is indeed a $k$-uniform displacement tableau.

\textit{Number of symbols in $t$.} We now enumerate the number of distinct symbols that occur in $t$ on $\lambda$. The primary difficulty in this computation is that, as we have constructed it, the symbols of $t$ are not consecutive; some symbols are skipped.

Every value $t(x,y)$ for $(x,y) \in \lambda$ has the form $(a - \ell) B + r$, where $1 \leq r \leq a - \ell$; clearly two such values are equal if and only if they arise from equal values of $B$ and $r$, so it suffices to enumerate the number of distinct pairs $(B,r)$ which occur.

The number $B$ is equal to $(x-1) + q(k + \ell - a)$. The number $q$ depends on $y$ alone; its possible values are the integers $0$ through $\left\lfloor \frac{a-1}{a - \ell} \right\rfloor$, inclusive. Note that the largest possible value of $q$ is at least $1$, since our assumptions imply that $\ell \geq 1$. For each possible value of $q$ in $\{0,1,\cdots, \left\lfloor \frac{a-1}{a - \ell} \right\rfloor - 1\}$, the possible values of $y$ giving this value of $q$ give all possible values of $r$ (from $1$ to $a - \ell$). Since $x-1$ takes all values from $0$ to $b-1$ inclusive, this shows that all pairs $(B,r)$ satisfying the following inequalities will occur.

\begin{eqnarray*}
0 &\leq B \leq& b-1 + \left( \left\lfloor \frac{a-1}{a - \ell} \right\rfloor - 1 \right) (k + \ell -a)\\
1 &\leq r \leq& a - \ell
\end{eqnarray*}

This accounts for 
\begin{equation} \label{monster1}
(a - \ell) \left( b + \left( \left\lfloor \frac{a-1}{a - \ell} \right\rfloor - 1 \right) (k + \ell -a) \right) 
\end{equation}
distinct values of $t(x,y)$.

All of the remaining values of $B$ must occur for the largest possible value of $q$, namely $q = \left\lfloor \frac{a-1}{a-\ell} \right\rfloor$. For this value of $q$, the possible values of $r$ range from $1$ to $a - \left\lfloor \frac{a-1}{a-\ell} \right\rfloor (a - \ell)$, inclusive. For this value of $q$, the value of $B$ is distinct from those listed above if and only if $x$ is taken to be one of the $(k+\ell-a)$ largest values from the range $\{1,2,\cdots,b\}$ (here we use the fact that $k+\ell-a \leq b$, which follows from the assumption $a+b-k \geq 2$ and the choice of $\ell$). Therefore the number of values of $t(x,y)$ not accounted for is precisely

\begin{equation} \label{monster2}
\left( a - \left\lfloor \frac{a-1}{a-\ell} \right\rfloor (a - \ell) \right) (k + \ell - a).
\end{equation}

Adding the quantities (\ref{monster1}) and (\ref{monster2}) and rearranging terms gives the quantity $(a-\ell)(b-\ell) + k \ell$, which is therefore the number of distinct symbols in $t$. It follows that $\cd(a,b,k) \leq (a-\ell)(b-\ell) + k \ell$, as desired.
\end{proof}

Lemmas \ref{l_cdlower} and \ref{l_cdupper} together prove Proposition \ref{p_cd}.

\begin{cor} \label{cor_tropdim}
If $k$ and $\ell$ are relatively prime, then the dimension of $W^r_d(\Gg)$ is equal to $\ovr$.
\end{cor}
\begin{proof}
By Riemann-Roch, we may assume without loss of generality that $d \leq g-1$. Then $\dim W^r_d(\Gg) = g - \cd(r+1,g-d+r,k)$ by Lemma \ref{l_td}, which is $g- \delta(r+1,g-d+r,k)$ by Proposition \ref{p_cd}. By definition, this is equal to $\ovr$.
\end{proof}

\section{Proofs of main theorems} \label{sec_proofs}

We can now combine our results from the previous two sections to prove the theorems from the introduction. We assume that Situation \ref{sit_main} holds throughout, i.e. integers $g,k,r,d$, a field $K$, and a general $k$-gonal curve $C$ are fixed, with the characteristic of $K$ meeting our mild restrictions. Before each proof we briefly recall the theorem statement.

Our main result is Theorem \ref{thm_main}, which asserts that in Situation \ref{sit_main}, $\ovr$ is an upper bound for $\dim W^r_d(C)$.

\begin{proof}[Proof of theorem \ref{thm_main}]
We first prove that the result holds when $K$ is a complete algebraically closed non-Archimedean field with nontrivial valuation. By Lemma \ref{lem_lexists}, there exists an integer $\ell$ such that $(\chr K, k, \ell)$ is admissible. By Lemma \ref{lem_lift} and Corollary \ref{cor_punchline}, there exists a smooth projective curve $C'$ over $K$ such that for all $r,d$,
$$
\dim W^r_d(C') \leq g - \cd(g-d+r,r+1,k),
$$
where $\cd$ is the function from Definition \ref{def_cd}. Proposition \ref{p_cd}, together with the fact that $\ovr = g - \delta(r+1,g-d+r,k)$, gives the inequality $\dim W^r_d(C') \leq \ovr$.

In particular, letting $r=1$, these bounds show that $W^1_d(C')$ is empty for all $d < k$ (the inequalities $d \leq k-1$ and $k \leq \frac{g+3}{2}$ imply that $\overline{\rho}_{g,k}(d,1) = d-k \leq -1$). So the gonality of $C'$ is at least $k$. By construction, $C'$ does have a degree $k$ line bundle with two linearly independent sections, so the gonality of $C'$ is exactly $k$. By the irreducibility of the $k$-gonal locus and upper semicontinuity, it follows that these upper bounds on dimension hold on a dense open subset of the set of $k$-gonal curves over the field $K$. Hence $\dim W^r_d(C) \leq \ovr$ for a \textit{general} $k$-gonal curve $C$ of genus $g$ over $K$.

\textit{Extension to arbitrary algebraically closed fields.} So far we have deduced the theorem only for a specific type of field. To obtain the result for all fields, observe that the coarse moduli space of curves of genus $g$ is a scheme of finite type over $\Spec \ZZ$, and the locus of $k$-gonal curves for which the theorem holds is a locally closed subscheme whose image in $\Spec \ZZ$ includes the generic point and $\Spec \textbf{F}_p$ for all $p$ not excluded by Situation \ref{sit_main}. Therefore it has points in every algebraically closed field of characteristic not excluded by Situation \ref{sit_main}.
\end{proof}

Theorem \ref{thm_cm} states that, under a characteristic $0$ hypothesis, $\unr$ is a lower bound for $\dim W^r_d(C)$; it follows from the main theorem of \cite{cm02} by an elementary argument. We give this elementary argument below for completeness. First, we restate the main theorem of \cite{cm02} in the notation of this paper.\\

\noindent\textbf{Main theorem of \cite{cm02}}. \textit{
Let $C$ be a general $k$-gonal curve of genus $g$ over the complex numbers, and let $r,d$ be positive integers. Let $\ell \in \{0,1,\cdots,r\}$ be an integer satisfying the following hypotheses.
\begin{enumerate}
\item $\ell \geq r- k$,
\item $r+1 - \ell$ divides either $r$ or $r+1$, and
\item $\rho_g(d,r-\ell) - \ell k \geq \max(0,\rho_g(d,r))$.
\end{enumerate}
Then $W^r_d(C)$ has an irreducible component of dimension $\rho_g(d,r-\ell) - \ell k$.
}

\begin{proof}[Proof of Theorem \ref{thm_cm}]
Note first of all that to prove this results for all algebraically closed fields of characteristic $0$, it suffices to verify the result for the field $K = \textbf{C}$ of complex numbers, by standard arguments (e.g. as in the last paragraph of the proof of Theorem \ref{thm_main}). We use the notation of Situation \ref{sit_main}. By Riemann-Roch, we may assume without loss of generality that $d \leq g-1$, so that $r' = r$ in the notation of Definition \ref{def_ovr}. Assume also that $\unr \geq 0$ and $\unr > \rho_g(d,r)$, since otherwise the desired result is either vacuous or follows from $\dim W^r_d(C) \geq \rho_g(d,r)$. Note that this implies $r > 0$.

Let $\ell$ be whichever integer $\ell \in \{0,1,r-1,r\}$ is closest to $\ell_0 = \frac{g-d+2r - k + 1}{2}$, where we choose the largest value in case of a tie. In this case $\rho_g(d,r-\ell) - \ell k = \unr$. So it suffices to verify hypotheses (1), (2), and (3) in the theorem statement above.

Hypothesis (3) follows from our assumptions. Hypothesis (2) holds for all $\ell \in \{0,1,r-1,r\}$ for elementary reasons.  It remains to verify hypothesis (1). If $\ell \in \{r-1,r\}$ then hypothesis (1) follows from $k \geq 2$. We are assuming that $\unr > \rho_g(d,r)$, hence $\ell \neq 0$; it remains to consider the case $\ell = 1$. We may assume that $r \geq 3$, otherwise either $r=0$, $\ell = 1 = r-1$ or $\ell = 1 = r$. Since $\ell = 1$ must be closer to $\ell_0$ than $r-1$ (which is strictly greater than $1$), it follows that $\ell_0 < \frac12 r$. From the definition of $\ell_0$, this implies that $r < k + d - g - 1$. We are assuming that $d < g$, hence $r < k - 1$. So certainly $1 \geq r - k$ in this case, and hypothesis (1) holds.

Therefore all three hypotheses hold, and the main theorem of \cite{cm02} implies that $W^r_d(C)$ has a component of dimension $\unr$.
\end{proof}

Theorem \ref{thm_eq} asserts that, in characteristic $0$, when $k \leq 5$ or $k \geq \frac15g+2$, the dimension of $W^r_d(C)$ is exactly $\ovr$ (unless this number is negative, in which case $W^r_d(C)$ is empty). The proof is elementary, given Theorems \ref{thm_main} and \ref{thm_cm}.

\begin{proof}[Proof of Theorem \ref{thm_eq}]
We must show that if $k \leq 5$ or $k \geq \frac15 g + 2$, then either $\unr = \ovr$ or $\ovr < 0$. We assume without loss of generality that $d \leq g-1$.

Suppose that $\unr \neq \ovr$. By the discussion in \ref{ss_gap},
$$
2 \leq \frac{g-d+2r-k+1}{2} \leq r-2.
$$

Define two auxiliary variables, as follows.
\begin{eqnarray*}
\ell_0 &=& \frac{g-d+2r-k+1}{2}\\
\delta &=& \frac{g-d-1}{2}
\end{eqnarray*}

We deduce from the inequalities above that $\ell_0 \geq 2$ and $\delta \leq \frac{k}{2} - 3$.  An elementary (but tedious) rearrangement of terms shows that, for any integer $\ell$,
$$
\rho_g(d,r-\ell) -k \ell = g - (\ell - \ell_0)^2 - \left( \frac{k}{2} \right)^2 - k \ell_0 + \delta^2.
$$
The maximum value occurs at $\ell = \lfloor \ell_0 \rfloor$, hence
\begin{eqnarray*}
\ovr &=& \rho_g(d,r-\lfloor \ell_0 \rfloor) - k \lfloor \ell_0 \rfloor\\
&\leq& g - (\lfloor \ell_0 \rfloor - \ell_0)^2 - \left( \frac{k}{2} \right)^2 - k \ell_0 + \delta^2\\
&\leq& g - 0 - \left( \frac{k}{2} \right)^2 - k \cdot 2 + \left( \frac{k}{2} - 3 \right)^2\\
&\leq& g - 5k + 9.
\end{eqnarray*}
In the case $k > \frac15 g + 2$, it follows that $\ovr \leq -1$. In the case $k \leq 5$, the inequality $\delta \leq \frac{k}{2} -3$ would imply that $\delta < 0$, which contradicts the assumption that $d \leq g-1$. Therefore we conclude that in either case, $\unr = \ovr$ unless $\ovr < 0$. The result now follows from Theorems \ref{thm_main} and \ref{thm_cm}.
\end{proof}

Theorem \ref{thm_howgeneral} characterizes, in characteristic $0$, those values of $d,r$ for which a general $k$-gonal curve of genus $g$ satisfies $\dim W^r_d(C) = \rho_g(d,r)$ (i.e. the same as for a general curve of genus $g$). This occurs if and only if $r=0,\ g-d+r = 1,$ or $g-k \leq d-2r$. 

\begin{proof}[Proof of Theorem \ref{thm_howgeneral}]
The $r=0$ case (and dually, the $g-d+r = 1$ case) follow from the fact that the $d$th symmetric product of $C$ surjects onto $W^0_d(C)$, so assume that $r \geq 1$ and $g-d+r \geq 2$. Let $\ell_0$ be as in Section \ref{ss_gap}. Then $\ell_0 \leq \frac12$ is equivalent to $g-k \leq d-2r$; the discussion in Section \ref{ss_gap} shows that in this case $\dim W^r_d(C) = \rho_g(d,r)$. For the converse, suppose that $g-k > d - 2r$; equivalently, $\ell_0 \geq 1$. Then $\rho_g(d,r-1) - k > \rho_g(d,r)$, hence $\unr > \rho_g(d,r)$. By Theorem \ref{thm_cm}, $\dim W^r_d(C) > \rho_g(d,r)$, as desired.
\end{proof}

\bibliography{main}{}
\bibliographystyle{alpha}

\end{document}